\documentclass[12pt]{article}
\usepackage[cp1251]{inputenc}
\usepackage{amsmath,amssymb}
\usepackage{amsthm}
\usepackage{mathrsfs}
\usepackage{tikz}
\usepackage{tikz-cd}
\usepackage{array,booktabs,float}
\usepackage{enumitem}
\usepackage{graphicx}
\usepackage{fancyhdr}

\theoremstyle{plain}
\newtheorem{theorem}{Theorem}[section]
\newtheorem{proposition}[theorem]{Proposition}
\newtheorem*{proposition*}{Proposition}
\newtheorem{lemma}[theorem]{Lemma}

\newtheorem*{syuteiri*}{Main Theorem}

\theoremstyle{definition}
\newtheorem{definition}[theorem]{Definition}

\newtheorem*{remark*}{Remark}

\newtheorem*{example*}{Example}

\newcommand{\R}{\mathbb{R}}
\newcommand{\C}{\mathbb{C}}
\newcommand{\abs}[1]{\left\lvert#1\right\rvert}

\newcommand{\all}[1]{\,{\vphantom{#1}}^\forall \!{#1}}

\newcommand{\setmid}{\mathrel{}\middle|\mathrel{}}

\newcommand{\eqdef}{\colon \mspace{-10mu} =}

\newcommand{\blfootnote}[1]{%
  \begingroup
  \renewcommand\thefootnote{}\footnote{#1}%
  \addtocounter{footnote}{-1}%
  \endgroup
}

\begin{document}

\markright{4-DIMENSIONAL TWISTOR SPACE}

\title{Almost contact structures on the set of rational curves in a 4-dimensional twistor space}
\author{MICHIFUMI TERUYA}
\date{}
\maketitle

\thispagestyle{fancy}
\renewcommand{\headrulewidth}{0pt}
\fancyhf{}
\fancyhead[C]{4-DIMENSIONAL TWISTOR SPACE}
\fancyhead[RE,LO]{}
\fancyfoot[C]{\thepage}

\begin{abstract}

In this paper, we provide a correspondence between certain 5-dimensional complex spacetimes and 4-dimensional twistor spaces.
The spacetimes are almost contact manifolds whose curvature tensor satisfies certain conditions.

By using the correspondence, we show that a 5-dimensional K-contact manifold can be obtained from the Ren-Wang twistor space\cite{RW}, which is obtained from two copies of $\C^4$ identifying open subsets by a holomorphic map.
From this result, the Ren-Wang twistor space can be interpreted in the framework of Itoh\cite{I}.

\end{abstract}

\tableofcontents

\section{Introduction}

 A 3-dimensional complex manifold $Z$ is called a \textit{(3-dimensional) twistor space} if there exists a rational curve $C \subset Z$ whose normal bundle $N_{C/Z}$ satisfies
\begin{equation}
	N_{C/Z} \cong \mathcal{O}(1) \oplus \mathcal{O}(1).
\end{equation}
In this situation, the rational curve $C$ is called a \textit{twistor line}.
By the deformation theory of submanifolds\cite{K}, since the equation
\begin{equation}
	H^1(C, N_{C/Z})=0
\end{equation}
holds, there is a complex analytic family $\{ C_x \}_{x \in M^\C}$, where $M^\C$ is a 4-dimensional complex manifold and $C_x$ is a holomorphic deformation of the twistor line $C$.
We think of the manifold $M^{\C}$ as the set of rational curves in $Z$ near the twistor line $C$.
We can endow the complex manifold $M^{\C}$ with a complex conformal class $[g^{\C}]$, where $g^\C$ is a complex metric.
The conformal class is self-dual.
Conversely, 3-dimensional twistor spaces arise from 4-dimensional self-dual complex manifolds \cite{P} \cite{P2}, which are called \textit{4-dimensional complex self-dual spacetimes}.
The correspondence between 3-dimensional twistor spaces and 4-dimensional self-dual manifolds is called the Penrose correspondence.

In this paper, a 4-dimensional complex manifold $\mathcal{P}$ is called a \textit{4-dimensional twistor space} if there exists a rational curve $C \subset \mathcal{P}$ whose normal bundle $N_{C/\mathcal{P}}$ satisfies
\begin{equation}
	N_{C/\mathcal{P}} \cong \mathcal{O} \oplus \mathcal{O}(1) \oplus \mathcal{O}(1).
\end{equation}
In this situation, the rational curve $C$ is called a \textit{twistor line}.
Since the equation
\begin{equation}
	H^1(C, N_{C/\mathcal{P}})=0
\end{equation}
holds, there is a complex analytic family $\{ C_x \}_{x \in \mathcal{M}^\C}$, where $\mathcal{M}^\C$ is a 5-dimensional complex manifold and $C_x$ is a holomorphic deformation of the twistor line $C$.
We want to give some geometric structures to the manifold $\mathcal{M}^{\C}$ and find a correspondence between 4-dimensional twistor spaces and some 5-dimensional complex manifolds as a kind of the Penrose correspondence.

Let $M$ be an odd-dimensional manifold.
A quadruple $(\phi,\xi,\theta,g)$ is called an \textit{almost contact structure} on the manifold $M$ if the $(1,1)$-tensor field $\phi$, the vector field $\xi$, the 1-form $\theta$, and the Riemannian metric $g$ satisfy certain properties (Definition \ref{def:21}).
A manifold equipped with an almost contact structure is called an \textit{almost contact manifold}.
If $(M,\phi,\xi,\theta,g)$ is a 5-dimensional almost contact manifold, then the distribution $E \eqdef \mathrm{Ker}\ \theta$ on $M$ is of rank 4.
We denote the dual space of $E$ by $E^*$.
The set of 2-forms $\wedge^2$ on the 5-dimensional manifold is decomposed into
\begin{equation}
	\wedge^2 = \wedge^2_+ (E^*) \oplus \wedge^2_{-}(E^*) \oplus (E^* \wedge \eta) , \label{eq:999}
\end{equation}
where $\wedge^2_+ (E^*)$ (resp. $\wedge^2_- (E^*)$) is the $(+1)$-eigenspace (resp. $(-1)$-eigenspace) of the star operator
\begin{equation}
	* \colon \wedge^2 E^* \to \wedge^2 E^*.
\end{equation}
We denote the space $\wedge^2_+ E^*$ (resp. $\wedge^2_- E^*$) by $\wedge_+$ (resp. $\wedge_-$), and $E^* \wedge \theta$ by $\wedge_0$.
According to the decomposition \eqref{eq:999}, the Riemannian curvature tensor $R \in \Gamma \left( \mathrm{End}(\wedge^2) \right)$ of $g$ can be written as the block matrix
\begin{equation}
	R =\begin{pmatrix} R^+_+ & R^+_- & R^+_0 \\
	R^-_+ & R^-_- & R^-_0 \\
	R^0_+ & R^0_- & R^0_0  \end{pmatrix}.
\end{equation}
The Weyl curvature $W \in \Gamma(\mathrm{End}(\wedge^2))$ can be similarly partitioned.

The main result is the following:

\begin{syuteiri*}
Let $\mathcal{P}$ be a 4-dimensional twistor space, $C$ a twistor line in $\mathcal{P}$, and $\mathcal{M}^\C$ a 5-dimensional complex manifold which is the set of holomorphic deformations of the twistor line $C$ in $\mathcal{P}$.
Then, the manifold $\mathcal{M}^\C$ has naturally associated almost contact structures.

There is the reverse construction:
if a 5-dimensional complex almost contact manifold satisfies the curvature conditions
\begin{equation}
	R^-_0 =0,\ \ \ W^-_- =0, \label{eq:000}
\end{equation}
then we can construct a 4-dimensional twistor space $\mathcal{P}$ as the set of ``$\beta$-surfaces'' in the almost contact manifold.
The almost contact manifold is one of the associated almost contact manifolds with the 4-dimensional twistor space $\mathcal{P}$.
\end{syuteiri*}

We give a brief account of the result.
The condition $W^-_- = 0$ corresponds to the self-duality of a 4-dimensional spacetime, which is the condition imposed on a spacetime in 3-dimensional twistor theory.
The condition $R^-_0 =0$ is specific to the 5-dimensional case.
In the main theorem, we assume that a spacetime is a complex almost contact manifold, but the associated twistor space does not depend on the $(1,1)$-tensor field $\phi$ since $\beta$-surfaces depend only on the metric $g$ and the 1-form $\theta$.
Thus, we do not have to impose the existence of $\phi$ on a 5-dimensional manifold.
However, important examples are special cases of almost contact manifolds. 

We illustrate an application of the main result.
Ren-Wang\cite{RW} construct a 4-dimensional complex manifold which is obtained from 5-dimensional complex Heisenberg group $\mathscr{H}$ as a set of ``$\alpha$-planes'' in $\mathscr{H}$.
Since this manifold is a 4-dimensional twistor space, we call it the Ren-Wang twistor space.
Ren-Wang characterize $\alpha$-planes by using the Lie group structure.
On the other hand, we characterize $\alpha$-planes by using an almost contact structure.
Hence, we can expect to extend the Penrose-Ward correspondence (\cite{RW}, Theorem 1.1) to 5-dimensional complex almost contact manifolds which satisfy certain conditions.

We describe the structure of the paper.

In Section 2, we explain the basics of almost contact structures.
For details, please refer to Sasaki\cite{S}, Hatakeyama-Sasaki\cite{HS}, and Blair\cite{B}.
We also define the equivalent class of almost contact structures.
This is the geometric structure that we give on the set of holomorphic deformations of a twistor line in a 4-dimensional twistor space.

In Section 3, we introduce some important concepts of 5-dimensional almost contact manifolds.
We define a \textit{$\beta$-plane} and a \textit{$\beta$-surface} in a 5-dimensional complex almost contact manifold.
We define on the bundle of $\beta$-planes some rank 2 distribution, which is called the \textit{tautological distribution}.
If this distribution is integrable, then we can construct the twistor space as the set of $\beta$-surfaces.
We give a sufficient condition for the integrability of the tautological distribution in Proposition \ref{prop:310}.

In Section 4, we define a 4-dimensional twistor space and prove the main theorem (Theorem \ref{thm:11}).
The correspondence between twistor spaces and spacetimes is similar to the case of 3-dimensional twistor spaces \cite{P2}.
Our proof is based on the method of LeBrun-Mason\cite{LM}.

In Section 5, we give two examples.
One is a trivial example:
if we consider a 3-dimensional twistor space $Z$ and set $\mathcal{P} \eqdef Z \times \C$, then the complex manifold $\mathcal{P}$ is a 4-dimensional twistor space.
In this case, the 1-form $\theta$ obtained in the main theorem satisfies $\theta \wedge d\theta=0$.
In particular, the complex almost contact structure obtained in the main theorem is not a complex K-contact structure.
The other is the ``standard example'':
the Ren-Wang twistor space, which is obtained from the 5-dimensional complex manifold $\mathscr{H}$, is an example of a 4-dimensional twistor space.
The equivalent class of complex almost contact structures obtained from this twistor space includes those which are complex K-contact structures.
Hence, we construct a complex K-contact structure on the 5-dimensional complex Heisenberg group $\mathscr{H}$.

In Section 6, we introduce the method of Itoh\cite{I}, which constructs a CR-twistor space from a 5-dimensional real K-contact manifold, and interpret the Ren-Wang twistor space in the framework of Itoh.

Gundry\cite{G} also considered 4-dimensional twistor spaces and constructed torsional Galilean structures on the sets of twistor lines.
In Appendix B, we show that in the case of the Ren-Wang twister space, the 1-form obtained by us coincides with the 1-form constituting the torsional Galilean structure obtained by Gundry.

\section*{Acknowledgments}

I would like to thank my Ph.D. supervisor, Professor Nobuhiro Honda, for his useful advice on how to proceed with the research and the thesis elaboration.
I would also like to thank Professor Maciej Dunajski for introducing Gundry's paper\cite{G} to me.

\section{Almost contact structure and K-contact structure}

In this section, we describe the notion of almost contact structures introduced in \cite{S}, \cite{HS}.
These structures are broader objects than contact structures.

\begin{definition}\label{def:21}
An \textit{almost contact structure} on an odd-dimensional manifold $M^{2n+1}$ is a quadruple $(\phi,\xi,\theta,g)$, where $\phi$ is a $(1,1)$-tensor field, $\xi$ is a vector field, $\theta$ is a 1-form, and $g$ is a Riemannian metric, which satisfy the following conditions:
\begin{align}
	\mathrm{rk} (\phi) &=2n, \label{eq:001} \\
	\phi(\xi) &=0, \label{eq:002} \\
	\theta(\phi(X)) &=0,\ \ \ \all{X} \in \mathfrak{X}(M), \label{eq:003} \\
	\phi^2 &=-\mathrm{id} +\theta \otimes \xi, \label{eq:004} \\
	\theta(\xi) &=1, \label{eq:005} \\
	\theta(X) &= g(X, \xi),\ \ \ \all{X} \in \mathfrak{X}(M), \label{eq:006} \\
	g(\phi(X),\phi(Y)) &= g(X, Y) -\theta(X) \theta(Y),\ \ \ \all{X,Y} \in \mathfrak{X}(M). \label{eq:007}
\end{align}

A manifold equipped with an almost contact structure is called an \textit{almost contact manifold}.
\hfill{$\Box$}
\end{definition}

By the notion of considering holomorphic tensors instead of real ones, we can define a \textit{complex almost contact structure} on a $(2n+1)$-dimensional complex manifold.

On a real or complex almost contact manifold $(M,\phi,\xi,\theta,g)$, there is the distribution
\begin{equation}
	E \eqdef \mathrm{Ker}(\theta) ;
\end{equation}
which is of rank $2n$, and the tangent bundle of $M$ is decomposed into
\begin{equation}
	TM= E \oplus \langle \xi \rangle. \label{eq:008}
\end{equation}
This decomposition is orthogonal with respect to the metric $g$.

Next, we define a \textit{contact metric manifold} and a \textit{K-contact manifold} as an almost contact manifold of a special kind:

\begin{definition}\label{def:22}
An almost contact structure $(\phi,\xi,\theta,g)$ on a manifold $M^{2n+1}$ is called a \textit{contact metric structure} if the 1-form $\theta$ is a contact form, the vector field $\xi$ is the Reeb vector field, and the equation
\begin{equation}
	g(X,\phi(Y)) = \frac{1}{2} d\theta(X,Y), \label{eq:010}
\end{equation}
holds for arbitrary vector fields $X,Y \in \mathfrak{X}(M)$.

A contact metric structure $(\phi,\xi,\theta,g)$ is called a \textit{K-contact structure} if the Reeb vector field $\xi$ is a Killing vector field with respect to the metric $g$.
\hfill{$\Box$}
\end{definition}

The above definition can be applied to a complex almost contact structure.
Thus, we can define a \textit{complex contact metric structure} and a \textit{complex K-contact structure}.

Next, to each almost contact structure, we define a special family of almost contact structures.

\begin{lemma}\label{lem:23}
Suppose that a quadruple $(\phi,\xi,\theta,g)$ is an almost contact structure on a manifold $M$.
For arbitrary non-vanishing functions $f$ and $F$, and an arbitrary vector field $X_0$ which belongs to $\Gamma(E)$, we set
\begin{align}
	\phi^\prime &\eqdef \phi -f \phi(X_0) \otimes \theta, \\
	\xi^\prime &\eqdef X_0 + f^{-1} \xi, \\
	\eta^\prime &\eqdef f \theta, \\
	g^\prime &\eqdef F g -fF g(\cdot,X_0) \otimes \theta -fF \theta \otimes g(\cdot,X_0) \notag \\
	&\ \ +\left( -F +f^2 +f^2 F g(X_0,X_0) \right) \theta \otimes \theta.
\end{align}
Then, the quadruple $(\phi^\prime, \xi^\prime, \theta^\prime, g^\prime)$ is also an almost contact structure on $M$.
\end{lemma}
\begin{proof}
Equations \eqref{eq:003} and \eqref{eq:005} are obvious, and the equation \eqref{eq:001} can be derived from other equations.
Hence, we only need to check the equations \eqref{eq:002},\eqref{eq:004},\eqref{eq:006}, and \eqref{eq:007}.

From direct calculation, it follows that
\begin{align}
	\phi^\prime \left( \xi^\prime \right) &= (\phi - f \phi(X_0) \otimes \theta)(X_0 +f^{-1} \xi) \notag \\
	&=\phi(X_0) -f \phi(X_0) \theta(f^{-1} \xi) =0. 
\end{align}
Thus, the equation \eqref{eq:002} holds. 
Also from direct calculation, for any vector field $X$, it follows that
\begin{align}
	(\phi^\prime)^2(X) &=(\phi - f \phi(X_0) \otimes \theta)^2(X) \notag \\
	&=(\phi - f \phi(X_0) \otimes \theta) (\phi(X) -f\theta(X) \phi(X_0)) \notag \\
	&=\phi^2(X) -f \theta(X) \phi^2(X_0) \notag \\
	&=(-\mathrm{id} +\xi \otimes \theta)(X) +(X_0 \otimes (f\theta)) (X) \notag \\
	&=\left( -\mathrm{id} + (f^{-1} \xi +X_0) \otimes \theta^\prime \right)(X) \notag \\
	&=\left( -\mathrm{id}+ \xi^\prime \otimes \theta^\prime \right) (X).
\end{align}
Thus, the equation \eqref{eq:004} holds.

Next, from the calculation
\begin{align}
	g^\prime \left( \xi^\prime, X \right) &= F g \left( \xi^\prime, X \right) -f F g \left( \xi^\prime, X_0 \right) \theta(X) -f F g (X,X_0) \theta \left( \xi^\prime \right) \notag \\
	&\ \ +(-F+f^2+f^2 F g(X_0,X_0)) \theta \left( \xi^\prime \right) \theta(X) \notag \\
	&=f^{-1}F \theta(X) +F g(X_0,X) -f F g(X_0,X_0) \theta(X) -F g(X,X_0) \notag \\
	&\ \ +(-F+f^2+f^2 F g(X_0,X_0))f^{-1} \theta(X) \notag \\
	&= f \theta(X) =\theta^\prime(X), 
\end{align}
we have \eqref{eq:006}.

Finally, noting that $\theta(\phi(X))=0$ for any vector field $X$, we have
\begin{equation}
	g^\prime(\phi(X),\phi(Y))= F g(\phi(X),\phi(Y)).
\end{equation}
Thus, we have
\begin{align}
	&\ \ g^\prime \left( \phi^\prime(X), \phi^\prime (Y) \right) \notag \\
	&= g^\prime (\phi(X) -f\theta(X) \phi(X_0), \phi(Y) -f\theta(Y) \phi(X_0)) \notag \\
	&= g^\prime (\phi(X),\phi(Y)) - f \theta(X) g^\prime (\phi(X_0),\phi(Y)) -f \theta(Y) g^\prime (\phi(X_0),\phi(X)) \notag \\
	&\ \ +f^2 \theta(X) \theta(Y) g^\prime(\phi(X_0),\phi(X_0)) \notag \\
	&=F g(\phi(X),\phi(Y)) - f F \theta(X) g (\phi(X_0),\phi(Y)) -f F \theta(Y) g (\phi(X_0),\phi(X)) \notag \\
	&\ \ +f^2 F \theta(X) \theta(Y) g(\phi(X_0),\phi(X_0)) \notag \\
	&=F g (X,Y) -F \theta(X) \theta(Y) -f F \theta(X) g (X_0,Y) -f F \theta(Y) g (X_0,X) \notag \\
	&\ \ +f^2 F \theta(X) \theta(Y) g (X_0,X_0) \notag \\
	&=g^\prime(X,Y) -f^2 \theta(X) \theta (Y) =g^\prime(X,Y) -\theta^\prime(X) \theta^\prime(Y), 
\end{align}
and the equation \eqref{eq:007} holds.
\end{proof}

\begin{definition}\label{def:24}
For an almost contact structure $(\phi,\xi,\theta,g)$ on a manifold $M$, we define a family of almost contact structures as
\begin{align}
	&[\phi,\xi,\theta,g] \notag \\
	&\eqdef \left\{ \left( \phi^\prime, \xi^\prime, \theta^\prime, g^\prime \right) \setmid f\ \text{and}\ F\ \text{are non-vanishing functions on $M$},\ \ X_0 \in \Gamma(E) \right\}, \notag
\end{align}
where the quadruple $\left( \phi^\prime, \xi^\prime, \theta^\prime, g^\prime \right)$ is the almost contact structure given in Lemma \ref{lem:23}.
The family $[\phi,\xi,\theta,g]$ is called the \textit{equivalent class} defined by $(\phi,\xi,\theta,g)$ or simply an equivalent class of almost contact structures.

We can also define an equivalent class of \textit{complex} almost contact structures in a similar way.
\hfill{$\Box$}
\end{definition}

\section{5-dimensional spacetimes}

In this section, we introduce some important concepts for 5-dimensional (real or complex) almost contact manifolds.

First, let $\mathcal{M}$ be a 5-dimensional (real or complex) manifold and a quadruple $(\phi,\xi,\theta,g)$ an almost contact structure on $\mathcal{M}$.
Then, we define the 4-dimensional distribution $E$ as
\begin{equation}
	E \eqdef \mathrm{Ker}\ \theta.
\end{equation}
The tangent bundle of $\mathcal{M}$ is decomposed into
\begin{equation}
	T \mathcal{M} = E \oplus \langle \xi \rangle,
\end{equation}
and the cotangent bundle of $M$ is decomposed into
\begin{equation}
	T^* \mathcal{M} = E^* \oplus \langle \theta \rangle. \label{eq:130}
\end{equation}
Since the bundle $E$ is of rank 4 and has equipped with the metric $g \rvert_E$, we can define the star operator on the space $\Gamma (\wedge^2 E^*)$, and we have
\begin{equation}
	*^2 =\mathrm{id}.
\end{equation}
We denote the space $\Gamma(E^*) \wedge \theta$ by $\wedge_0$.
In this situation, the set of 2-forms $\wedge^2$ is decomposed into
\begin{equation}
	\wedge^2 =\wedge_+ \oplus \wedge_{-} \oplus \wedge_0, \label{eq:131}
\end{equation}
where $\wedge^2_\pm $ is the $(\pm 1)$-eigenspace of the star operator respectively.
The star operator can be extended on $\wedge^2$ by setting it to $0$ on $\wedge_0$.
Using the decomposition \eqref{eq:131}, the curvature tensor $R \in \mathrm{End}(\wedge^2)$ can be written as the block matrix
\begin{equation}
	R =\begin{pmatrix} R^+_+ & R^+_- & R^+_0 \\
	R^-_+ & R^-_- & R^-_0 \\
	R^0_+ & R^0_- & R^0_0  \end{pmatrix}. 
\end{equation}
The Weyl curvature $W$ and the trace-free Ricci tensor $K$ can be similarly partitioned.

\begin{lemma}\label{lem:31}
Let $K \in \mathrm{End} (\wedge^2)$ be the trace-free Ricci tensor.
Then,
\begin{equation}
	K \circ * +* \circ K \equiv s^\prime * \ \ \ (\mathrm{mod}\ \wedge_0) \label{eq:1311}
\end{equation}
holds on $\wedge^2 E^*$, where $s^\prime$ is a function determined solely from the metric $g$.
\end{lemma}
\begin{proof}
Let a tuple of 1-forms $(\theta^1,\dots,\theta^4)$ be a locally orthonormal frame of $E^*$ such that $\theta^1 \wedge \dots \wedge \theta^4$ is a volume form of $g \rvert_E$.
If we denote the 1-form $\theta$ by $\theta^0$, then $(\theta^0,\theta^1,\dots,\theta^4)$ is a locally orthonormal frame of the cotangent bundle $T^* \mathcal{M}$.
Since we can rearrange the order of the 1-forms, it is enough to show that the equation \eqref{eq:1311} holds for the 2-form $\theta^1 \wedge \theta^2$.

By the definition of the trace-free Ricci tensor $K$, the equation
\begin{equation}
	K(\theta^i \wedge \theta^j) =\frac{2}{3} \left( \sum_{k=0}^5 K_{ki} \theta^k \wedge \theta^j + K_{kj} \theta^i \wedge \theta^k \right)
\end{equation}
holds, where $K_{ij}$ is the component of $K$.
From the direct calculation, we have
\begin{align}
	*K(\theta^1 \wedge \theta^2) &= \frac{2}{3}* (K_{22} \theta^1 \wedge \theta^2 +K_{32} \theta^1 \wedge \theta^3 +K_{42} \theta^1 \wedge \theta^4 \notag \\
	&\ \ + K_{11} \theta^1 \wedge \theta^2 +K_{31} \theta^3 \wedge \theta^2 +K_{41} \theta^4 \wedge \theta^2 ) \notag \\
	&=\frac{2}{3} (K_{22} \theta^3 \wedge \theta^4 -K_{32} \theta^2 \wedge \theta^4 + K_{42} \theta^2 \wedge \theta^3 \notag \\
	&\ \ +K_{11} \theta^3 \wedge \theta^4 -K_{31} \theta^1 \wedge \theta^4 +K_{41} \theta^1 \wedge \theta^3),
\end{align}
\begin{align}
	K(* (\theta^1 \wedge \theta^2)) &= K(\theta^3 \wedge \theta^4) \notag \\
	&\equiv \frac{2}{3} ( K_{13} \theta^1 \wedge \theta^4 +K_{23} \theta^2 \wedge \theta^4 +K_{33} \theta^3 \wedge \theta^4 \notag \\
	&\ \ +K_{14} \theta^3 \wedge \theta^1 +K_{24} \theta^3 \wedge \theta^2 +K_{44} \theta^3 \wedge \theta^4) \ \ (\mathrm{mod}\ \wedge_0).
\end{align}
Since $K_{ij}$ is symmetric, we have
\begin{equation}
	(K \circ * +* \circ K) \theta^1 \wedge \theta^2 \equiv \frac{2}{3} \sum_{i=1}^4 K_{ii} *(\theta^1 \wedge \theta^2) \ \ \ (\mathrm{mod}\ \wedge_0).\label{eq:90004}
\end{equation}
In general, in the 5-dimensional case, the right-hand side of the equation \eqref{eq:90004} does not vanish but depends only on the metric $g$.
\end{proof}

\begin{lemma}\label{lem:32}
For the component of trace-free Ricci tensor $K^-_-$, the equation
\begin{equation}
	K^-_- = \frac{s^\prime}{2} \mathrm{id}
\end{equation}
holds, where $s^\prime$ is the function given in Lemma \ref{lem:31}.
\end{lemma}
\begin{proof}

Let $\varphi$ be an anti-self-dual 2-form.
According to the decomposition \eqref{eq:131}, we decompose the 2-form $K\varphi$ into
\begin{equation}
	K\varphi = \varphi_+ +\varphi_- + \varphi_0.
\end{equation}
Then, we have
\begin{equation}
	* K\varphi = \varphi_+ -\varphi_-. \label{eq:1234}
\end{equation}
On the other hand, by Lemma \ref{lem:31}, we see that
\begin{align}
	* K\varphi &\equiv - K (* \varphi) +s^\prime * \varphi \notag \\
	&\equiv K \varphi -s^\prime \varphi \notag \\
 	&\equiv \varphi _+ +(\varphi_- -s^\prime \varphi) \ \ \ (\mathrm{mod} \ \wedge_0).  \label{eq:1235}
\end{align}
By comparing anti-self-dual components of \eqref{eq:1234} and \eqref{eq:1235}, we have
\begin{equation}
	K^-_- \varphi =\varphi_- = \frac{s^\prime}{2} \varphi. 
\end{equation}
\end{proof}

Next, we consider a 5-dimensional complex almost contact manifold

\noindent
$(\mathcal{M}^\C, \phi, \xi, \theta, g )$.
For any point $x \in \mathcal{M}^\C$, we identify the 4-dimensional inner product space $(E_x,g_x)$ with the complex Euclidian space $(\C^4,g_{\text{eucl}})$, and decompose it into the spinor spaces
\begin{equation}
	E_x \cong S_x^+ \otimes S_x^- .
\end{equation}

\begin{definition}\label{def:33}
A 2-dimensional subspace $V$ in the tangent space $T_x \mathcal{M}^\C$ is called an \textit{$\alpha$-plane} (resp. a $\beta$-plane) if $V \subset E_x$ and $V$ can be written as
\begin{equation}
	V = \psi^+ \otimes S^-_x,\ \ \ \psi^+ \in S_x^+\ \ \ (\text{resp.}\ \  V=S^-_x \otimes \psi^{-},)
\end{equation}
for some $\psi^+ \in S_x^+ \setminus \{ 0 \}$ (resp. $\psi^- \in S_x^- \setminus \{ 0 \}$).
Also, a 2-dimensional subspace $V$ in $T_x M$ is called \textit{null} if the restriction of the metric $g \rvert_V$ is $0$. \hfill{$\Box$}
\end{definition}

From the same argument as in the four-dimensional case\cite{LM}, we see that:

\begin{lemma}
Any $\alpha$-planes and $\beta$-planes are null planes.
Conversely, any null plane in $E_x$ is an $\alpha$-plane or a $\beta$-plane.
Also, a null element in $\wedge_-$ corresponds to a $\beta$-plane.
\end{lemma}

The bundle of $\beta$-planes
\begin{equation}
	\widetilde{\mathcal{P}} \eqdef \bigcup_{x \in \mathcal{M}^\C} \mathbb{P}(S^-_x)
\end{equation}
is a holomorphic $\mathbb{CP}^1$-bundle over the 5-dimensional manifold $\mathcal{M}^\C$.
We denote the projection $\widetilde{\mathcal{P}} \to \mathcal{M}^\C$ by $\pi$.

\begin{definition}
Let $\widetilde{\mathcal{E}}$ be the distribution of rank 2 on the complex manifold $\widetilde{\mathcal{P}}$ defined by
\begin{equation}
	\widetilde{\mathcal{E}}_{\tilde{z}} \eqdef L_{\tilde{z}} (\mathcal{B}_{\tilde{z}}),\ \ \ \tilde{z} \in \widetilde{\mathcal{P}},
\end{equation}
where $\mathcal{B}_{\tilde{z}}$ is the $\beta$-plane corresponding to $\tilde{z} \in \widetilde{ \mathcal{P}}$ and the map $L_{\tilde{z}}$ is the horizontal lift
\begin{equation}
	L_{\tilde{z}} \colon T_{\pi(\tilde{z})} \mathcal{M}^\C \to T_{\tilde{z}} \widetilde{\mathcal{P}}
\end{equation}
with respect to the Levi-Civita connection of $g$.
We call this distribution $\widetilde{\mathcal{E}}$ the \textit{tautological distribution}.
\hfill{$\Box$}
\end{definition}

\begin{definition}[\cite{LM}, Definition 3.2]
A complex surface $\mathcal{S} \subset \mathcal{M}^\C$ is called a \textit{proto-$\beta$-surface} if the tangent space $T_x \mathcal{S}$ is a $\beta$-surface for any point $x \in \mathcal{S}$. 
Also, a proto-$\beta$-surface $\mathcal{S}$ is said to be \textit{maximal} if the surface is not a proper subset of any proto-$\beta$-surface.
A maximal proto-$\beta$-surface is called a \textit{$\beta$-surface}.
\hfill{$\Box$}

\end{definition}

The horizontal lift of a $\beta$-surface is an integral surface of the tautological distribution $\widetilde{\mathcal{E}}$. 
Therefore, it follows that:

\begin{lemma}\label{lem:37}
The following are equivalent:
\begin{enumerate}
\item The tautological distribution $\widetilde{\mathcal{E}}$ is integrable.
\item For any $x \in \mathcal{M}^\C$ and any $\beta$-plane $\mathcal{B}$ in $T_x \mathcal{M}^\C$, there exists a $\beta$-surface through $x$ whose tangent space at $x$ is $\mathcal{B}$.
\end{enumerate}
\end{lemma}

To give a sufficient condition for the integrability of the tautological distribution $\widetilde{\mathcal{E}}$, we show the following lemmas.

\begin{lemma}\label{lem:38}
Assume that a 5-dimensional complex almost contact manifold $(\mathcal{M}^\C,\phi,\xi,\theta,g)$ satisfies
\begin{equation}
	R^0_- =0,\ \ \ W^-_-=0.
\end{equation}
Then, for any $\beta$-plane $\mathcal{B}$, if $v,w \in \mathcal{B}$, then $R(v,w)v \in \mathcal{B}$.
\end{lemma}
\begin{proof}

The condition $R(v,w)v \in \mathcal{B}$ is satisfied if and only if the equations
\begin{align}
	g(R(v,w)v, v) &= 0, \label{eq:102002} \\
	g(R(v,w)v, w) &= 0, \label{eq:102003} \\
	g(R(v,w)v, \xi) &= 0 \label{eq:102004}
\end{align}
hold since $\mathcal{B} =(\mathcal{B} \oplus \langle \xi \rangle)^\perp $.
From the Bianchi identity, the equation \ref{eq:102002} always holds.
Hence, it is enough to show that the equations \eqref{eq:102003} and \eqref{eq:102004} hold for any $\beta$-plane $\mathcal{B}$.

For two vectors $v, w \in \mathcal{B}$, the wedge product $v \wedge w$ can be identified a null element in $(\wedge_-)^*$.
Thus, from Lemma \ref{lem:32}, we have
\begin{align}
	g(R(v,w)v,w) &=- \left( R(v \wedge w) ,v \wedge w \right) \notag \\
	&=- \left( W(v \wedge w), v \wedge w \right).
\end{align}
Hence, from the assumption $W^-_-=0$, we have
\begin{equation}
	g(R(v,w)v,w) =0.
\end{equation}
Similarly, from the assumption $R^0_- =0$, we have
\begin{equation}
	g(R(v,w)v,\xi) =0.
\end{equation}
Therefore, we obtain that
\begin{equation}
	R(v,w)v \in ( \langle v,w,\xi \rangle )^\perp =\mathcal{B},
\end{equation}
as desired.
\end{proof}

\begin{remark*}
The reverse claim of Lemma \ref{lem:38} does not hold, i.e.
the condition $R(v,w)v \in \mathcal{B}$ for any $v,w \in \mathcal{B}$ does not mean $R^0_- =W^-_- =0$.
For the details, see Appendix A.
\hfill{$\Box$}
\end{remark*}

\begin{lemma}\label{lem:39}
Let $(\mathcal{M}^\C,\phi,\xi,\theta,g)$ be a 5-dimensional complex almost contact manifold.
For any $\beta$-plane $\mathcal{B}$, we assume that $R(v,w)v \in \mathcal{B}$ for any two vectors $v,w$ tangent to the $\beta$-plane $\mathcal{B}$.
Then, for any $\beta$-plane $\mathcal{B}$, there exists a $\beta$-surface $\mathcal{S}$ which is tangent to $\mathcal{B}$.

Also, if we assume that $(\mathcal{M}^\C,\phi,\xi,\theta,g)$ is a complex K-contact manifold or the distribution $\mathrm{Ker}\ \theta$ is integrable, then the converse claim holds.
\end{lemma}
\begin{proof}
We omit the proof of the existence of the $\beta$-surface because it can be shown in a similar to the 4-dimensional case.
See \cite{LM} for the details.

Next, for any $\beta$-plane $\mathcal{B}$, we assume that there exists a $\beta$-surface $\mathcal{S}$ which is tangent to $\mathcal{B}$.
It is enough to show that the covariant derivative $\nabla_v w$ is in $\mathcal{B}$ for any two vectors $v,w \in \mathcal{B}$.
Since the $\beta$-plane $\mathcal{B}$ is a tangent space of the integral surface $\mathcal{S}$, we have
\begin{equation}
	\nabla_v w -\nabla_w v =[v,w] \in \mathcal{B}.
\end{equation}
Since $\beta$-planes are null planes, we have
\begin{align}
	g(\nabla_v w, v) &= g(\nabla_w v, v) = \frac{1}{2} w(g(v,v)) =0, \\
	g(\nabla_v w, w) &= v(g(w,w)) =0.
\end{align}
Hence, we need to show that
\begin{equation}
	g(\nabla_v w, \xi) =0. \label{eq:133}
\end{equation}

First, we assume that the distribution $\mathrm{Ker}\ \eta$ is integrable.
In this case, the 5-dimensional manifold $\mathcal{M}^\C$ is decomposed into a direct product.
Therefore, the equation \eqref{eq:133} holds.

Finally, we assume that $(\mathcal{M}^\C,\phi,\xi,\theta,g)$ is a complex K-contact manifold.
For any K-contact manifold,
\begin{equation}
	\nabla_v \xi =-\phi(v)
\end{equation}
holds (cf. \cite{B}).
Thus, we have
\begin{align}
	g(\nabla_v w, \xi) &= -g(w, \nabla_v \xi ) \notag \\
	&=g(w,\phi(v)) \notag \\
	&=\frac{1}{2} d\theta (w,v) \notag \\
	&= -\frac{1}{2} \theta([w,v])=0.
\end{align}
Therefore, the equation \eqref{eq:133} holds.
\end{proof}

From Lemma \ref{lem:37}, Lemma \ref{lem:38}, and Lemma \ref{lem:39}, the following proposition immediately follows:

\begin{proposition}\label{prop:310}
Assume that a 5-dimensional complex almost contact manifold satisfies
\begin{equation}
	R^0_- =0,\ \ \ W^-_-=0.
\end{equation}
Then, the tautological distribution $\widetilde{\mathcal{E}}$ is integrable.
\end{proposition}

\section{4-dimensional twistor spaces}

In this section, we show the correspondence between 4-dimensional twistor spaces and certain 5-dimensional complex almost contact manifolds.
To this end, we define 4-dimensional twistor spaces.

\begin{definition}\label{def:41}
A 4-dimensional complex manifold $\mathcal{P}$ is called a \textit{4-dimensional twistor space} if there exists a rational curve $C \subset \mathcal{P}$ whose normal bundle satisfies  
\begin{equation}
	N_{C/\mathcal{P}} \cong \mathcal{O} \oplus \mathcal{O}(1) \oplus \mathcal{O}(1).
\end{equation}
In this situation, the rational curve $C$ is called a \textit{twistor line}.
\hfill{$\Box$}
\end{definition}

We apply the deformation theory of submanifolds\cite{K} to a twistor line in a 4-dimensional twistor space.
Since the 1st cohomology group of the normal bundle of a twistor line $C$ vanishes, there is a complex analytic family $\{ C_x \}_{x \in \mathcal{M}^\C}$, where $\mathcal{M}^\C$ is a 5-dimensional complex manifold.
Also, the tangent space of $\mathcal{M}^\C$ at a twistor line $C$ is
\begin{equation}
	T_C \mathcal{M}^\C \cong H^0(C,N_{C/\mathcal{P}}) \cong H^0(C, \mathcal{O} \oplus \mathcal{O}(1) \oplus \mathcal{O}(1) )\cong \C^5. \label{eq:3333}
\end{equation}
By considering the ``null cone'' of the tangent space, we can define  almost contact structures on $\mathcal{M}^\C$.
We define the set $\mathcal{N}_x$ as
\begin{equation}
	\mathcal{N}_x \eqdef \left\{ v \in T_x \mathcal{M}^\C \setmid \text{The section}\ s_v\ \text{has a zero.} \right\} \label{eq:60606}
\end{equation}
for each point $x \in \mathcal{M}^\C$, where the section $s_v \in H^0(C,N_{C/\mathcal{P}})$ is the corresponding section to the vector $v$.

\begin{lemma}\label{lem:42}
In the above situation, the set $\mathcal{N}_x$ defines a quadratic cone in a hyperplane in the tangent space $T_x \mathcal{M}^\C$.
Thus, the defining functions of $\mathcal{N}_x$ are a linear polynomial and a quadratic polynomial on $T_x \mathcal{M}^\C$.
The linear polynomial defines the 1-form $\theta$ and the quadratic polynomial defines the degenerate bilinear form $g_0$.
\end{lemma}
\begin{proof}
We regard a vector $v \in T_x \mathcal{M}^\C$ as a holomorphic section $s_v$ of the holomorphic vector bundle by using the isomorphism \eqref{eq:3333}.
The holomorphic section $s_v$ can be written as
\begin{equation}
	s_v=((s_v)_1,(s_v)_2) \in H^0 (C_x, \mathcal{O}) \oplus H^0(C_x, \mathcal{O}(1) \oplus \mathcal{O}(1)),
\end{equation}
where the holomorphic section $(s_v)_1 \in H^0 (C_x, \mathcal{O}) \cong \C$ is a constant and $(s_v)_2$ is a holomorphic section of $H^0(C_x, \mathcal{O}(1) \oplus \mathcal{O}(1))$.
The vector $v$ is in $\mathcal{N}_x$ if and only if $(s_v)_1 \equiv 0$ and $(s_v)_2$ has a zero.
The former condition defines a linear polynomial $\theta$ on $T_x \mathcal{M}^\C$.
The latter condition is given by the determinant of a $2 \times 2$ matrix. 
Thus, the latter condition defines a quadratic form $g_0$ on $T_x \mathcal{M}^\C$.
Since the quadratic form $g_0$ vanishes on $H^0 (C_x, \mathcal{O})$, $g_0$ is degenerate. 
\end{proof}

\begin{remark*}
For a holomorphic tensor field $\Psi$, we define the \textit{class of holomorphic tensor fields $[\Psi]$} as
\[
	[\Psi] \eqdef \left\{ f\Psi \setmid \text{$f$ is a non-vanishing holomorphic function} \right\}.
\]
Let a 1-form $\theta$ and a symmetric bilinear form $g_0$ be tensors given in Lemma \ref{lem:42}.
For arbitrary non-vanishing functions $f$ and $F$ on $\mathcal{M}^\C$, the 1-from $f \theta$ and the symmetric bilinear form $F g_0$ defines the same null cone $\mathcal{N}_x$.
Thus, Lemma \ref{lem:42} gives the class of 1-forms $[\theta]$ and the class of bilinear forms $[g_0]$ on $\mathcal{M}^\C$.
\hfill{$\Box$}
\end{remark*}

\begin{proposition}\label{thm:01}
Let $\mathcal{P}$ be a 4-dimensional twistor space, $C$ a twistor line in $\mathcal{P}$, and $\mathcal{M}^\C$ the above 5-dimensional complex manifold.
Then, we can construct the equivalent class of complex almost contact structures on $\mathcal{M}^\C$.
\end{proposition}
\begin{proof}
From Lemma \ref{lem:42}, we have the 5-dimensional manifold $\mathcal{M}^\C$, the 1-form $\theta$, and the bilinear form $g_0$.
Replacing the space $\mathcal{M}^\C$ with a sufficiently small neighborhood of a point if necessary, $\theta$ and $g_0$ do not vanish on $\mathcal{M^\C}$.
The symmetric form $g_0$ is non-degenerate on the subbundle $\mathrm{Ker}(\theta)$.
We define the symmetric bilinear form $g$ as
\begin{equation}
	g \eqdef \theta^2 +g_0. \label{eq:3338}
\end{equation}
It is non-degenerate.
There is a unique holomorphic vector field $\xi \in (\mathrm{Ker}\ \theta)^\perp$ such that
\begin{equation}
	\theta(\xi)=1.
\end{equation}
We take a holomorphic $(1,1)$-tensor field $\phi$ on $\mathcal{M}^\C$ such that the equations \eqref{eq:002} to \eqref{eq:004} hold.
Such a tensor field $\phi$ always exists if the manifold $\mathcal{M}^\C$ is sufficiently small.
Then, the quadruple $(\phi,\xi,\theta,g)$ is a complex almost contact structure.
The equivalent class of complex almost contact structures in the sense of Definition \ref{def:24} $[\phi,\xi,\theta,g]$ is defined by $(\phi,\xi,\theta,g)$.
\end{proof}

\begin{remark*}
Let a 1-form $\theta$ and a symmetric bilinear form $g_0$ be tensors given in Lemma \ref{lem:42}.
For arbitrary non-vanishing functions $f$ and $F$ on $\mathcal{M}^\C$, we get the 1-from $f \theta$ and the symmetric bilinear form $F g_0$.
Applying the above method in proof of Proposition \ref{thm:01} to the 1-form $f \theta$ and the symmetric bilinear form $F g_0$, we get a complex almost contact structure $(\phi^\prime, \xi^\prime, \theta^\prime, g^\prime)$ on $\mathcal{M}^\C$ (see Lemma \ref{lem:23}).
This complex almost contact structure is contained in the equivalent class $[\phi,\xi,\theta,g]$.
\hfill{$\Box$}
\end{remark*}

Next, we show the reverse construction.
Let $(\mathcal{M}^\C, \phi, \xi, \theta, g)$ be a 5-dimensional complex almost contact manifold.
If we assume that the tautological distribution on $\mathcal{M}^\C$ is integrable, then there is a foliation on $\widetilde{\mathcal{P}}$ whose leaves are horizontal lifts of $\beta$-surfaces.
By replacing $\mathcal{M}^\C$ with a sufficiently small neighborhood of a point if necessary, we assume that the leaf space $\mathcal{P}$ of this foliation is a 4-dimensional complex manifold.
Then, we get the following double fibration (Figure \ref{fig:02}), where the map $\tau \colon \widetilde{\mathcal{P}} \to \mathcal{M}^\C$ is the projection of the bundle and the map $\eta$ is the quotient mapping.

\begin{figure}[ht]
	\centering
		\includegraphics{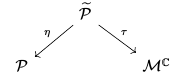}
	\caption{double fibration}\label{fig:02}
\end{figure}

We define a submanifold $C_x \subset \mathcal{P}$ as
\begin{equation}
	C_x \eqdef \eta \circ \tau^{-1} (x) \ (x \in \mathcal{M}^\C).
\end{equation}
The following proposition shows that the normal bundle of $C_x$ is isomorphic to the bundle $\mathcal{O} \oplus \mathcal{O} \oplus \mathcal{O}(1)$.

\begin{proposition}\label{thm:02}
In the above situation, the complex manifold $\mathcal{P}$ is a 4-dimensional twistor space.
Also, the set
\begin{equation}
	C_x \eqdef \eta \circ \tau^{-1}(x),\ \ \ (x \in \mathcal{M}^\C)
\end{equation}
is a twistor line.
\end{proposition}
\begin{proof}

Let $x \in \mathcal{M}^\C$ be a point.
The set $\widetilde{C_x} \eqdef \tau^{-1}(x) \cong \mathbb{CP}^1$ consists of $\beta$-planes in the tangent space $T_x \mathcal{M}^\C$. 
If $\beta$-planes $\tilde{z}, \tilde{z}^\prime \in \widetilde{C_x}$ satisfy
\[
	\eta(\tilde{z}) = \eta ( \tilde{z}^\prime ),
\]
then there is a $\beta$-surface which is tangent to $\tilde{z}$ and $\tilde{z}^\prime$.
Since the manifold $\mathcal{M}^\C$ is replaced with a sufficiently small neighborhood of a point, we can assume that a $\beta$-surface has no self-intersection.
Hence, $\tilde{z} = \tilde{z}^\prime$ holds.
Thus, the holomorphic map $\eta \rvert_{\widetilde{C_x}}$ is a bijective to the image.
Hence, the set $C_x \eqdef \eta \circ \tau^{-1} (x)$ is a rational curve.

Next, we show that the normal bundle $N_{C_x/\mathcal{P}}$ is isomorphic to the vector bundle $\mathcal{O} \oplus \mathcal{O}(1) \oplus \mathcal{O}(1)$.
To show that, we use a local twistor description \cite{WW}.
In 3-dimensional twistor theory, a tangent vector $Q$ at a point $z$ in a 3-dimensional twistor space can be thought of as a connecting vector joining the $\beta$-surface corresponding to $z$ and the neighboring $\beta$-surface.
By using this, a vector field on a twistor line can be expressed as a pair of spinors $(\Omega^A,\Pi_{A^\prime})$.
The pair of spinors is called a local twistor representing. 
The normal bundle of a twistor line in a 3-dimensional twistor space is isomorphic to $\mathcal{O}(1) \oplus \mathcal{O}(1)$ since a vector field tangent to the twistor line corresponds to a pair of spinors $(0,\Pi_{A^\prime})$ and $\Omega^{A}\ (A=0,1)$ is homogeneous of degree $1$ in homogeneous coordinates of the twistor line.

In 4-dimensional case, the tangent space $T_x \mathcal{M}^\C$ is decomposed into 
\begin{equation}
	T_x \mathcal{M}^\C = \langle \xi_x \rangle \oplus E_x.
\end{equation}
A vector field of a rational curve $C_x$ in a 4-dimensional twistor space can be expressed as a triple $(\Xi,\Omega^A,\Pi_{A^\prime})$, where $\Xi$ is the $\langle \xi_x \rangle$-component. 
This is a local twistor representing in the 4-dimensional case.
Since the component $\Xi$ is homogeneous of degree $0$ in homogeneous coordinates of $C_x$, the normal bundle $N_{C_x/\mathcal{P}}$ satisfies
\begin{equation}
	N_{C_x/\mathcal{P}} \cong \mathcal{O} \oplus \mathcal{O}(1) \oplus \mathcal{O}(1).
\end{equation}
\end{proof}

We call the above manifold $\mathcal{P}$ the 4-dimensional twistor space associated with the 5-dimensional complex almost contact manifold $(\mathcal{M}^\C,\phi,\xi,\theta,g)$.  

Let $\mathcal{P}$ be the 4-dimensional twistor space associated with a 5-dimensional complex almost contact manifold $(\mathcal{M}^\C,\phi,\xi,\theta,g)$.
We assume that there is a complex analytic family $\{ C_x \}_{x \in \mathcal{M}^\C}$, where $C_x = \eta \circ \tau^{-1}(x) \subset \mathcal{P}$ is a twistor line.
The twistor line $C_x$ is a set of $\beta$-surfaces through $x$.
Let $v$ be a tangent vector in $T_x \mathcal{M}^\C$.
The vector $v$ defines a holomorphic deformation of the twistor line $C_x$ in $\mathcal{P}$.
We denote this deformation by $C_{x +v}$.
A zero of the section $s_v \in H^0 (C_x, N_{C_x/ \mathcal{P}})$ corresponding to $v$ is a point in $C_x \cap C_{x+v}$.
Let $z$ be a point in $C_x \cap C_{x+v}$.
The point $z$ can be thought of as a $\beta$-surface through $x$.
Since the point $z$ is in $C_{x+v}$, the vector $v$ tangents to this $\beta$-surface.
Hence, the vector $v$ is in the null cone $\mathcal{N}_x$ defined by the equation \eqref{eq:60606} if and only if $v$ is null with respect to the complex almost contact structure $(\phi,\xi,\theta,g)$.
Therefore, Proposition \ref{thm:01} and Proposition \ref{thm:02} are the reverse construction of each other.

From the above, we show the following:
\begin{theorem}\label{thm:11}
We can construct a 5-dimensional complex almost contact manifold from a 4-dimensional twistor space.

Conversely, if a 5-dimensional complex almost contact manifold satisfies
\begin{equation}
	W^-_- =0,\ \ \ R^0_- =0,
\end{equation}
then we can construct a 4-dimensional twistor space.
\hfill{$\Box$}
\end{theorem}

\section{Examples}

In this section, we describe two types of examples of Theorem \ref{thm:11}.
One is a trivial example, which is a trivial line bundle over a 3-dimensional twistor space.
The other is the Ren-Wang twistor space, which is investigated in \cite{RW}.

\subsection{Trivial examples and flat model}

Let $Z$ be a 3-dimensional twistor space and $C$ a twistor line in $Z$.
There is a complex analytic family $\{ C_x \}_{x \in M^\C}$, where $M^\C$ is a 4-dimensional manifold and each $C_x$ is a twistor line.
We consider the trivial bundle
\begin{equation}
	\mathcal{P} = Z \times \C
\end{equation}
over $Z$ and the 5-dimensional manifold
\begin{equation}
	\mathcal{M}^\C = M^\C \times \C.
\end{equation}
For each point $(x,t) \in \mathcal{M}^\C$, we set
\begin{equation}
	C_{(x,t)} = C_x \times \{ t \} \subset \mathcal{P}.
\end{equation}
The normal bundle of $C_{(x,t)}$ is isomorphic to $\mathcal{O} \oplus \mathcal{O}(1) \oplus \mathcal{O}(1)$.
Hence, the manifold $\mathcal{P}$ is a 4-dimensional twistor space, and there is a complex analytic family $\{ C_{(x,t)} \}_{(x,t) \in \mathcal{M}^\C}$.
By Lemma \ref{lem:42}, we have the class of degenerate symmetric form $[g_0]$ on $\mathcal{M}^\C$.
On the other hand, by the theory of 3-dimensional twistor spaces, there is the self-dual conformal class $[g_1]$ on $M^\C$.
If we pull back $[g_1]$ to $\mathcal{M}^\C$ by the projection $\mathcal{M}^\C \to M^\C$, then $[g_0] = [g_1]$. 
Also, if $t$ is a coordinate on $\C$, then the class of 1-forms provided by Lemma \ref{lem:42} is $[dt]$, which consists of $dt$ multiplied by non-vanishing functions. 
From the above, the metric $g=g_0+dt^2$, which is a complex metric on $\mathcal{M}^\C$, is the product metric.

\subsection{Ren-Wang twistor space}

Next, we deal with the 4-dimensional complex manifold which is investigated in \cite{RW}.
We call this manifold the Ren-Wang twistor space.
Ren-Wang twistor space is defined as the set of ``$\alpha$-planes'' in the 5-dimensional complex Heisenberg group.
The definition of an $\alpha$-plane in \cite{RW} is given by the structure of the 5-dimensional complex Heisenberg group as a Lie group, not an almost contact manifold.
Thus, we have to show that the Ren-Wang twistor space is an example of Theorem \ref{thm:11}. 

First, we describe the Ren-Wang twistor space.
The space is obtained from two copies $W$ and $\widetilde{W}$ of $\C^4$ by identifying their open subsets by the biholomorphic map
\begin{align}
	\Phi &\colon W \setminus \{ \zeta=0 \} \to \widetilde{W} \setminus \{ \widetilde{\zeta} =0 \}, \notag \\
	&\ \ (\zeta,\omega_0,\omega_1,\omega_2) \mapsto (\zeta^{-1},\zeta^{-1}\omega_0,\zeta^{-1}\omega_1,\omega_2+2\zeta^{-1}\omega_0 \omega_1), \label{eq:3336}
\end{align}
where $(\zeta,\omega_0,\omega_1,\omega_2)$ are coordinates on $W$ and $(\widetilde{\zeta},\widetilde{\omega}_0,\widetilde{\omega}_1,\widetilde{\omega}_2)$ are coordinates on $\widetilde{W}$.
We write the resulting 4-dimensional complex manifold by $\mathcal{P}$ and call it the Ren-Wang twistor space.

Let $(y_{AA^\prime},t)\ (A,A^\prime=0,1)$ be coordinates on $\C^5$ and maps $\eta$ and $\tau$ locally given by
\begin{align}
	\eta &\colon (y_{AA^\prime}, t , \zeta) \mapsto (\zeta, y_{00^\prime}+\zeta y_{01^\prime}, y_{10^\prime} +\zeta y_{11^\prime}, t- \langle \mathbf{y}_{0^\prime} +\zeta \mathbf{y}_{1^\prime}, \mathbf{y}_{1^\prime} \rangle), \notag \\
	\tau &\colon (y_{AA^\prime}, t , \zeta) \mapsto (y_{AA^\prime},t), \label{eq:036}
\end{align}
where $\mathbf{y}_{A^\prime} =(y_{0A^\prime}, y_{1A^\prime}) \ (A^\prime=0,1)$ and 
\begin{equation}
	\langle \mathbf{w}, \widetilde{\mathbf{w}} \rangle = w_1 \widetilde{w}_2 +\widetilde{w}_1 w_2
\end{equation}
for two vectors $\mathbf{w}=(w_1,w_2), \widetilde{\mathbf{w}} =(\widetilde{w}_1, \widetilde{w}_2) \in \C^2$.
There is a double fibration (Fig. \ref{fig:01})\cite{RW}.
\begin{figure}[ht]
	\centering
		\includegraphics{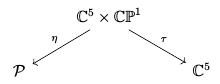}
	\caption{double fibration}\label{fig:01}
\end{figure}

Let $\pi$ be a natural projection $\pi \colon \mathcal{P} \to \mathbb{CP}^1$ given by
\begin{align}
	&\pi \rvert_W \colon (\zeta, \omega_0, \omega_1, \omega_2) \mapsto \zeta, \notag \\
	&\pi \rvert_{\widetilde{W}} \colon (\widetilde{\zeta},\widetilde{\omega}_0,\widetilde{\omega}_1,\widetilde{\omega}_2) \mapsto \widetilde{\zeta}.
\end{align}
The map $\pi \colon \mathcal{P} \to \mathbb{CP}^1$ gives a $\C^3$-bundle over a complex projective line.
The rational curve $C_0 \eqdef \eta \circ \tau^{-1} (\mathbf{0},0)$ is the zero section of the bundle.

\begin{lemma}\label{lem:05}
The normal bundle of the rational curve $C_0 \subset \mathcal{P}$ satisfies
\begin{equation}
	N_{C_0 /\mathcal{P}} \cong \mathcal{O} \oplus \mathcal{O}(1) \oplus \mathcal{O}(1).
\end{equation}
Therefore, the Ren-Wang twistor space is a 4-dimensional twistor space, and $C_0$ is a twistor line.
\end{lemma}
\begin{proof}
To find the normal bundle, we calculate the Jacobian $J_{\Phi}$ of the map $\Phi$ defined in \eqref{eq:3336}.
From the direct calcuration, it is
\begin{equation}
	J_{\Phi} = \begin{pmatrix} -\zeta^{-2} & 0 & 0 & 0 \\
	-\zeta^{-2}\omega_0 & \zeta^{-1} & 0 & 0 \\
	-\zeta^{-2}\omega_1 & 0 & \zeta^{-1} & 0 \\
	-2\zeta^{-2} \omega_0 \omega_1 &2\zeta^{-1} \omega_1& 2\zeta^{-1} \omega_0& 1 \end{pmatrix}.
\end{equation}
Hence, the Jacobian restricted to $C_0 = \{ \omega_0=\omega_1=\omega_2=0 \}$ is
\begin{equation}
	J_{\Phi} \rvert_{C_0} =\begin{pmatrix} -\zeta^{-2} & 0 & 0 & 0 \\
	0 & \zeta^{-1} & 0 & 0 \\
	0 & 0 & \zeta^{-1} & 0 \\
	0 & 0& 0& 1 \end{pmatrix}.
\end{equation}
This is the transition function of $\iota^* T\mathcal{P}$, where the map $\iota \colon C_0 \to \mathcal{P}$ is the natural inclusion.
Therefore, the normal bundle $N_{C_0/\mathcal{P}}$ is
\begin{equation}
	N_{C_0 /\mathcal{P}} = \iota^* T \mathcal{P} / TC_0 \cong \mathcal{O} \oplus \mathcal{O}(1) \oplus \mathcal{O}(1).
\end{equation}
\end{proof}

By taking suitable coordinates, we can show that $C_{(y_{AA^\prime},t)} \eqdef \eta \circ \tau^{-1} (y_{AA^\prime},t)$ is a twistor line for any $(y_{AA^\prime},t) \in \C^5$.
So any rational curve of this form is a twistor line, then the space $\mathcal{M}^\C$ of twistor lines in the 4-dimensional twistor space $\mathcal{P}$ is exactly $\C^5$.

Next, we construct an equivalent class of almost contact structures on the manifold $\mathcal{M}^\C = \C^5$.
For a tangent vector
\begin{equation}
	v =\sum_{A,A^\prime} a_{AA^\prime} \frac{\partial}{\partial y_{AA^\prime}} +b\frac{\partial}{\partial t} \in T_{(y_{AA^\prime},t)} \mathcal{M}^\C,
\end{equation}
we set
\begin{equation}
	C_{(y_{AA^\prime},t)+v} \eqdef \eta \circ \tau^{-1} ((y_{AA^\prime},t)+(a_{AA^\prime},b)).
\end{equation}
The vector $v$ is in the null cone $\mathcal{N}$ defined in Lemma \ref{lem:42} if and only if the equation
\begin{equation}
	C_{(y_{AA^\prime},t)} \cap C_{(y_{AA^\prime},t)+v} \neq \emptyset \label{eq:100025}
\end{equation}
holds.
In this situation, we find a necessary and sufficient condition for \eqref{eq:100025}.

\begin{lemma}\label{lem:06}
Under the above situation, the equation
\begin{equation}
	C_{(y_{AA^\prime},t)} \cap C_{(y_{AA^\prime},t)+v} \neq \emptyset
\end{equation}
holds if and only if
\begin{align}
	&\mathrm{det}(a_{AA^\prime})=0,\label{eq:023} \\
	&b -y_{00^\prime} a_{11^\prime} -y_{10^\prime}a_{01^\prime}+y_{01^\prime}a_{10^\prime} +y_{11^\prime}a_{00^\prime}=0 \label{eq:024}
\end{align}
holds.
\end{lemma}
\begin{proof}
The intersection of the rational curve $C_{(y_{AA^\prime},t)}$ and the coordinate neighborhood $W$ is given by
\begin{equation}
	C_{(y_{AA^\prime},t)} \cap W = \left\{ (\zeta, y_{00^\prime}+\zeta y_{01^\prime}, y_{10^\prime} +\zeta y_{11^\prime}, t- \langle \mathbf{y}_{0^\prime} +\zeta \mathbf{y}_{1^\prime}, \mathbf{y}_{1^\prime} \rangle) \setmid \zeta \in \C \right\}.
\end{equation}
Hence, $C_{(y_{AA^\prime},t)} \cap C_{(y_{AA^\prime},t)+v} \neq \emptyset$ holds if and only if there exists $\zeta \in \C$ which satisfies
\begin{align}
	y_{00^\prime}+\zeta y_{01^\prime} &=y_{00^\prime}+\zeta y_{01^\prime} +a_{00^\prime} +\zeta a_{01^\prime},  \\
	y_{10^\prime} +\zeta y_{11^\prime} &= y_{10^\prime} +\zeta y_{11^\prime} +a_{10^\prime} +\zeta a_{11^\prime},  \\
	 t- \langle \mathbf{y}_{0^\prime} +\zeta \mathbf{y}_{1^\prime}, \mathbf{y}_{1^\prime} \rangle &= t+b- \langle \mathbf{y}_{0^\prime} +\zeta \mathbf{y}_{1^\prime}+\mathbf{a}_{0^\prime} +\zeta \mathbf{a}_{1^\prime}, \mathbf{y}_{1^\prime} + \mathbf{a}_{1^\prime} \rangle.
\end{align}
Removing the common terms from both sides, we get
\begin{align}
	&a_{00^\prime} +\zeta a_{01^\prime} =0, \label{eq:020} \\
	&a_{10^\prime} +\zeta a_{11^\prime} =0, \label{eq:021} \\
	&b-\langle \mathbf{y}_{0^\prime} +\zeta \mathbf{y}_{1^\prime}+\mathbf{a}_{0^\prime} +\zeta \mathbf{a}_{1^\prime}, \mathbf{a}_{1^\prime} \rangle - \langle \mathbf{a}_{0^\prime} +\zeta \mathbf{a}_{1^\prime}, \mathbf{y}_{1^\prime} \rangle =0. \label{eq:022} 
\end{align}
There exists $\zeta \in \C$ which satisfies equations \eqref{eq:020} and \eqref{eq:021} if and only if
\begin{equation}
	\mathrm{det} (a_{AA^\prime}) =0
\end{equation}
holds, and if this is the case, then
\begin{equation}
	\mathbf{a}_{0^\prime}+\zeta \mathbf{a}_{1^\prime} =0
\end{equation}
holds.
Therefore, the equation \eqref{eq:022} turns out to be
\begin{align}
	0 &= b-\langle \mathbf{y}_{0^\prime} +\zeta \mathbf{y}_{1^\prime}, \mathbf{a}_{1^\prime} \rangle \notag \\
	&=b- \langle \mathbf{y}_{0^\prime} , \mathbf{a}_{1^\prime} \rangle - \langle \zeta \mathbf{y}_{1^\prime}, \mathbf{a}_{1^\prime} \rangle \notag \\ 
	&=b- \langle \mathbf{y}_{0^\prime} , \mathbf{a}_{1^\prime} \rangle +\langle \mathbf{y}_{1^\prime}, \mathbf{a}_{0^\prime} \rangle \notag \\
	&=b- y_{00^\prime} a_{11^\prime} -y_{10^\prime}a_{01^\prime}+y_{01^\prime}a_{10^\prime} +y_{11^\prime}a_{00^\prime}. 
\end{align}
\end{proof}

From the equation \eqref{eq:024} of Lemma \ref{lem:06}, the class of 1-forms on $\mathcal{M}^\C$ determined by Lemma \ref{lem:42} is
\begin{equation}
	[dt+y_{11^\prime} dy_{00^\prime} -y_{10^\prime}dy_{01^\prime} +y_{01^\prime}dy_{10^\prime}-y_{00^\prime}dy_{11^\prime}],
\end{equation}
and, from the equation \eqref{eq:023}, the class of degenerate symmetric forms determined by Lemma \ref{lem:42} is
\begin{equation}
	[dy_{00^\prime}dy_{11^\prime}-dy_{01^\prime}dy_{10^\prime}].
\end{equation}
We now show that there exist special elements in these classes which give a K-contact structure on $\mathcal{M}^\C$. 

Choose representatives in these classes as
\begin{align}
	\theta &= dt+y_{11^\prime} dy_{00^\prime} -y_{10^\prime}dy_{01^\prime} +y_{01^\prime}dy_{10^\prime}-y_{00^\prime}dy_{11^\prime}, \label{eq:6008} \\
	g_0 &= -2\sqrt{-1} (dy_{00^\prime}dy_{11^\prime}-dy_{01^\prime}dy_{10^\prime}).
\end{align}
In this case, the 1-form $\theta$ is non-degenerate, and the Reeb vector field $\xi$ of $\theta$ is given by
\begin{equation}
	\xi = \frac{\partial}{\partial t}.
\end{equation}
Also,  as we did in \eqref{eq:3338}, we define a non-degenerate metric $g$ on $\mathcal{M}^\C$ by
\begin{equation}
	g \eqdef \theta^2+g_0.
\end{equation}

Next, we find a $(1,1)$-tensor field $\phi$ such that the quadruple $(\phi,\xi,\theta,g)$ is a contact metric structure.
First, we introduce a local frame of the tangent bundle $T \mathcal{M}^\C$.
For simplicity, we introduce new indices $i \ (1 \leq i \leq 4)$ as
\begin{equation}
	i \eqdef 2A+A^\prime+1,
\end{equation}
and set
\begin{align}
	e_1 &= \partial_1 -y^4 \partial_t,\ \ e_2 = \partial_2 + y^3 \partial_t, \notag \\
	e_3 &= \partial_3 -y^2 \partial_t, \ \ e_4 =\partial_4 +y^1 \partial_t,\ \ e_0 =\xi= \partial_t,
\end{align}
where
\begin{equation}
	\partial_i = \frac{\partial}{\partial y^i} \ \ \ (i=1,2,3,4),\ \ \ \partial_t =\frac{\partial}{\partial t}.
\end{equation}
The frame of $T^* \mathcal{M}^\C$ which is dual to $e_0,\dots,e_4$ is given by
\begin{equation}
	\theta^0=\theta,\ \ \ \theta^i = dy^i\ \ \ (i=1,2,3,4).
\end{equation}
We write down the condition for $\phi=(\phi^i_j)$ to define a $(1,1)$-tensor $\phi$ which consists of a contact metric structure in terms of these frames.
First, from \eqref{eq:002} and \eqref{eq:003}, we get
\begin{align}
	0 &= \theta (\phi) =\phi^0_j, \\
	0 &=\phi(\xi) = \phi^i_0.
\end{align}
Also, by
\begin{equation}
	d\theta =-2 (\theta^1 \wedge \theta^4 -\theta^2 \wedge \theta^3 ),
\end{equation}
we have
\begin{align}
	\frac{1}{2} d\theta(e_i, -\phi(e_j)) &=(\theta^1 \wedge \theta^4 -\theta^2 \wedge \theta^3 )\left( e_i, \phi^k_j e_k \right) \notag \\
	&=\delta^1_i \phi^4_j -\delta^4_i \phi^1_j -\delta^2_i \phi^3_j +\delta^3_i \phi^2_j
\end{align}
and
\begin{align}
	g_0(e_i,e_j) &= -\sqrt{-1}(\theta^1 \otimes \theta^4 + \theta^4 \otimes \theta^1 - \theta^2 \otimes \theta^3 - \theta^3 \otimes \theta^2)(e_i,e_j) \notag \\
	&=-\sqrt{-1} (\delta^1_i \delta^4_j + \delta^4_i \delta^1_j -\delta^2_i \delta^3_j -\delta^3_i \delta^2_j).
\end{align}
Therefore, from \eqref{eq:010}, $\phi$ needs to be
\begin{equation}
	\phi= \begin{pmatrix} 0 & & & & \\
	& -\sqrt{-1} & & & \\
	& & -\sqrt{-1} & & \\
	& & & \sqrt{-1} & \\
	& & & & \sqrt{-1} \end{pmatrix}.
\end{equation}
This $\phi$ satisfies the conditions \eqref{eq:004} and \eqref{eq:007}, and the quadruple $(\phi,\xi,\theta,g)$ is a contact metric structure.
Since any components of $g$ do not depend on $t$, this contact metric structure is a K-contact structure on $\mathcal{M}^\C$.

\begin{remark*}
We endow $\C^5$ with the above K-contact structure.
In this situation, if $V$ is an $\alpha$-plane in the sense of Definition \ref{def:33}, then $V$ is an $\alpha$-plane in the sense of \cite{RW}, and vice versa.
\hfill{$\Box$}

\end{remark*}

\section{Ren-Wang twistor space as an example of Itoh\cite{I}}

In this section, first, we describe the results of Itoh \cite{I}, which constructs a 7-dimensional real CR manifold and a 4-dimensional complex manifold from a 5-dimensional real K-contact manifold.
Next, we show that the Ren-Wang twistor space, which is described in the previous section, can be interpreted in the framework of \cite{I}.

First, let  $(\mathcal{M},\phi,\xi,\theta,g)$ be a 5-dimensional real K-contact manifold.
We define the 7-dimensional real manifold $\mathcal{Z}$ as
\[
	\mathcal{Z} \eqdef \bigcup_{x \in \mathcal{M}} \left\{ \alpha \in (\wedge_{-})_x \setmid \abs{\alpha}=2 \right\},
\]
where the norm $\abs{\alpha}$ is given by the metric on the set of 2-forms $\wedge^2$ induced by $g$.
We denote the projection $\mathcal{Z} \to \mathcal{M}$ by $\pi$.
Since, the manifold $\mathcal{Z}$ is a $S^2$-bundle over the Riemannian manifold $(\mathcal{M},g)$, the tangent bundle $T \mathcal{Z}$ is decomposed into
\begin{equation}
	T \mathcal{Z} =\mathcal{V} \oplus \mathcal{H}^\prime \oplus \mathcal{H}^{\prime \prime},
\end{equation}
where $\mathcal{V}$ is the vertical bundle, $\mathcal{H}^\prime$ is the horizontal lift of $E$, and $\mathcal{H}^{\prime \prime}$ is the horizontal lift of $\langle \xi \rangle$.

We define a natural almost CR structure on $\mathcal{J}$ by the following way.
A point in $\mathcal{Z}$ is a tuple $(x, \alpha)$.
Since the fiber of $\mathcal{Z} \to \mathcal{M}$ can be thought of as $\mathbb{CP}^1$, there is the natural almost complex structure $j$ on the vertical subspace $\mathcal{V}_{(x,\alpha)}$.
We define $\mathcal{J}$ on $\mathcal{V}_{(x, \alpha)}$ as
\begin{equation}
	\mathcal{J}_{(x, \alpha)} = j\ \ \ \text{on}\ \mathcal{V}_{(x, \alpha)}.
\end{equation}
The 2-form $\alpha$ can be thought of as a negative spinor i.e. an almost complex structure on $E_x$.
We denote this almost complex structure by $J_\alpha$.
Thus, we define $\mathcal{J}$ on $\mathcal{H}^\prime_{(x, \alpha)}$ as
\begin{equation}
	\mathcal{J}_{(x, \alpha)} = (\pi_*)^{-1} (J_\alpha (\pi_*))\ \ \ \text{on}\ \mathcal{H}^\prime_{(x, \alpha)},
\end{equation}
where $\pi_* \colon \mathcal{H}^\prime_{(x, \alpha)} \to E_x$ is the inverse of the horizontal lift.
Finally, we define $\mathcal{J}$ on $\mathcal{H}^{\prime \prime}_{(x, \alpha)}$ as
\begin{equation}
	\mathcal{J}_{(x, \alpha)} =0 \ \ \ \text{on} \ \mathcal{H}^{\prime \prime}_{(x, \alpha)}.
\end{equation}
From the above method, we get the almost CR structure $\mathcal{J}$ on the manifold $\mathcal{Z}$.
This almost CR structure is integrable if and only if $W^-_- =0$ and the scalar curvature $s =-4$ (\cite{I}, Theorem 1.1).
If this almost CR structure is integrable, then the manifold $\mathcal{Z}$ is called a \textit{CR twistor space}.

We define a 8-dimensional real manifold $\mathcal{P}^\prime$ as
\begin{equation}
	\mathcal{P}^\prime = \mathcal{Z} \times I,
\end{equation}
where $I$ is an interval.
The tangent bundle $T \mathcal{P}^\prime$ is decomposed into
\begin{equation}
	T \mathcal{P}^\prime = \mathcal{V} \oplus \mathcal{H}^\prime \oplus \mathcal{H}^{\prime \prime} \oplus \left\langle \frac{\partial}{\partial t} \right\rangle,
\end{equation}
where $t$ is a coordinate on the interval $I$.
We define a almost complex structure $\mathcal{I}$ by the following way.
A point $p$ in $\mathcal{P}^\prime$ is a tuple $(x,\alpha,t)$, where $x$ is a point in $\mathcal{M}$, $\alpha \in (\wedge_-)_x$, and $t$ is a point in $I$.
We define $\mathcal{I}$ on $\mathcal{V}_p \oplus \mathcal{H}^\prime_p$ as
\begin{equation}
	\mathcal{I}_p = \mathcal{J}_{(x,\alpha)} \ \ \ \text{on}\ \mathcal{V}_p \oplus \mathcal{H}^\prime_p,
\end{equation}
and on $\mathcal{H}^{\prime \prime}_p \oplus \langle (\partial/\partial t) \rangle_p$ as
\begin{equation}
	a \widehat{\xi}_p + b \left(\frac{\partial}{\partial t} \right)_t \mapsto -b \widehat{\xi}_p + a \left(\frac{\partial}{\partial t} \right)_t,
\end{equation}
where $\widehat{\xi}$ is the horizontal lift of the Reeb vector field $\xi$.

\begin{theorem}[\cite{I}, Corollary 1.1]\label{thm:50}
In the above situation, if the metric $g$ satisfies the curvature conditions
\begin{align}
	\mathcal{W}^-_- &=0, \label{eq:053} \\
	\mathcal{R}^-_0 &=0, \label{eq:054} \\
	s &=-4, \label{eq:055}
\end{align}
then the almost complex structure $\mathcal{I}$ on $\mathcal{P}^\prime$ is integrable, where $s$ is the scalar curvature of $g$.
\hfill{$\Box$}
\end{theorem}

Next, we choose a real slice of the 5-dimensional complex K-contact manifold derived from the Ren-Wang twistor space and show that the curvature conditions in Theorem \ref{thm:50} are satisfied.
By taking suitable factors, we set
\begin{align}
	\widetilde{\theta} = \frac{\sqrt{-1}}{2} \theta, \\
	\widetilde{g_0} =\frac{\sqrt{-1}}{2} g_0.
\end{align}
In this situation, if we set
\begin{align}
	\widetilde{\xi} &= -2\sqrt{-1} \xi, \\
	\widetilde{\phi} &= \phi, 
\end{align}
then the quadruple $\left( \widetilde{\phi},\widetilde{\xi},\widetilde{\theta},\widetilde{g} \right)$ is a complex K-contact structure on $\mathcal{M^\C}$.
We define the inclusion $\R^5 \to \C^5 =\mathcal{M}^\C$ as
\begin{equation}
	(x^1,x^2,x^3,x^4,s) \mapsto \left( \begin{pmatrix} x^1 + \sqrt{-1}x^2 & -x^3 +\sqrt{-1}x^4 \\
	x^3 +\sqrt{-1}x^4 & x^1 -\sqrt{-1}x^2 \end{pmatrix}, -2\sqrt{-1}s \right),
\end{equation}
and denote the image of this inclusion by $\mathcal{M}^\R$.
By restricting the complex K-contact structure on $\mathcal{M}^\C$ to $\mathcal{M}^\R$, we get tensors
\begin{align}
	\widetilde{\theta} \rvert_{\mathcal{M}^\R} &= ds +x^2 dx^1 -x^1 dx^2 -x^4 dx^3 +x^3 dx^4, \label{eq:040} \\
	\widetilde{g} \rvert_{\mathcal{M}^\R} &= (\tilde{\theta} \rvert_{\mathcal{M}^\R})^2 +(dx^1)^2 +(dx^2)^2 +(dx^3)^2 +(dx^4)^2, \label{eq:041} \\
	\widetilde{\xi} \rvert_{\mathcal{M}^\R} &= \frac{\partial}{\partial s}, \label{eq:042} \\
	\widetilde{\phi} \rvert_{\mathcal{M}^\R} &= \begin{pmatrix} 0 & & & & \\
	& 0 & 1& & \\
	& -1 & 0 & & \\
	& & & 0& -1 \\
	& & & 1& 0 \end{pmatrix}, \label{eq:043}
\end{align}
where we use the frame
\begin{align}
	e_1 &= \partial_1 -x^2 \partial_s,\ \ e_2 =\partial_2 +x^1 \partial_s, \notag \\
	e_3 &= \partial_3 +x^4 \partial_s,\ \ e_4 =\partial_4 -x^3 \partial_s,\ \ e_0 = \partial_s, \label{eq:056}
\end{align}
and the dual frame
\begin{equation}
	\theta^0 =\widetilde{\theta}\rvert_{\mathcal{M}^\R} ,\ \ \ \theta^i =dx^i\ \ \ (i=1,2,3,4). \label{eq:057}
\end{equation}
to write $(1,1)$-tensor field $\widetilde{\phi} \rvert_\mathcal{M^\R}$ as a matrix.
The quadruple $\left( \widetilde{\phi} , \widetilde{\xi},\widetilde{\theta} ,\widetilde{g} \right)$ is a real K-contact structure on the 5-dimensional real manifold $\mathcal{M}^\R \cong \R^5$.
Considering the coordinate transformation
\begin{align}
	Z &= 2s +2x^1 x^2 -2x^3 x^4, \notag \\
	Y_1 &=2x^1,\ \ X_1 = 2x^2, \ \ Y_2 =2x^3,\ \ X_2 = -2x^4, 
\end{align}
we find that the above K-contact structure on $\R^5$ is the standard one.

We take the frame \eqref{eq:056} of the tangent bundle and the dual frame \eqref{eq:057} on $\mathcal{M}^\R \cong \R^5$.
By using these frames, the metric is given by 
\begin{equation}
	g_{ij} =\delta_{ij}.
\end{equation}
If we define commutation coefficients $c_{ij}^{\ \ k}$ by
\begin{equation}
	[e_i,e_j]= c_{ij}^{\ \ k} e_k,
\end{equation}
then, from the definition of $e_i$, we get
\begin{equation}
	c_{ij}^{\ \ k} = \begin{cases} 2 & \left( k=0\ \text{and}\ (i,j)=(1,2),(4,3) \right), \\
	-2 & \left( k=0\ \text{and}\ (i,j)=(2,1),(3,4) \right), \\
	0 & (\text{otherwise}). \end{cases} \label{eq:058}
\end{equation}
The Christoffel symbols $\Gamma^k_{ij}$ are given by
\begin{align}
	\Gamma^k_{ij} &= \frac{1}{2} g^{kl} \left( g_{lj,i}+g_{il,j}-g_{ij,l}+c_{lji}+c_{lij}+c_{ijl} \right) \notag \\
	&=\frac{1}{2}(c_{kj}^{\ \ i}+c_{ki}^{\ \ j}+c_{ij}^{\ \ k}).\label{eq:059}
\end{align}
Hence, we have
\begin{equation}
	\Gamma^k_{ij} = \begin{cases} \frac{1}{2}c_{kj}^{\ \ 0} & (i=0), \\
		\frac{1}{2} c_{ki}^{\ \ 0} & (j=0), \\
		\frac{1}{2} c_{ij}^{\ \ 0} & (k=0), \\
		0 & (\text{otherwise}), \end{cases}
\end{equation}
where $g_{lj,i}$ means the derivative of the function $g_{lj}$ by the vector field $e_i$.
In this situation, if we define the curvature tensor $R^l_{\ kij}$ as 
\begin{equation}
	R^l_{\ kij}e_l = R(e_i,e_j)e_k = \nabla_i \nabla_j e_k - \nabla_j \nabla_i e_k -\nabla_{[e_i,e_j]} e_k,
\end{equation}
then we have
\begin{equation}
	R^l_{\ kij} =\Gamma^m_{jk} \Gamma^l_{im} -\Gamma^m_{ik} \Gamma^l_{jm} -c_{ij}^{\ \ m} \Gamma^l_{mk},
\end{equation}
since the Christoffel symbols are constant.
Since the metric and commutation coefficients are invariant under the permutation $\sigma=(1,4)(2,3) \in S_5$, components of curvature tensor which are non-zero are 
\begin{align}
	R_{0101} &=1, \ \ R_{1212} = -3,\ \ R_{1234} = 2,\ \ R_{1324} =1,\ \ R_{1423} =-1, \notag \\
	R_{0202} &=1, \ \ R_{2314} =-1,\ \ R_{2413} =1. \label{eq:060}
\end{align}
From these, for the Ricci curvature
\begin{equation}
	R_{ij} =\sum_k R_{ikjk},
\end{equation}
we obtain
\begin{equation}
	\left( R_{ij} \right) =\begin{pmatrix} 4 & & & & \\
		& -2 & & & \\
		& & -2 & & \\
		& & & -2 & \\
		& & & & -2 \end{pmatrix}, \label{eq:061}
\end{equation}
and the scalar curvature $s$,
\begin{equation}
	s=-4. \label{eq:062}
\end{equation}

If we give the orientation to the distribution $E$ by setting the 4-form
\[
	\theta^1 \wedge \theta^2 \wedge \theta^4 \wedge \theta^3
\]
to be positive
(i.e. we give the opposite orientation to the standard one), then the space of self-dual 2-forms and the space of anti-self-dual 2-forms are spanned by
\begin{align}
	&\theta^1 \wedge \theta^2 -\theta^3 \wedge \theta^4,\ \ \theta^1 \wedge \theta^3 + \theta^2 \wedge \theta^4,\ \ \theta^1 \wedge \theta^4 - \theta^2 \wedge \theta^3, \notag \\
	&\theta^1 \wedge \theta^2 +\theta^3 \wedge \theta^4,\ \ \theta^1 \wedge \theta^3 - \theta^2 \wedge \theta^4,\ \ \theta^1 \wedge \theta^4 + \theta^2 \wedge \theta^3, \notag
\end{align}
respectively.
Also, the space of 2-forms including the contact form, which is $\wedge_0$, is spanned by
\[
	\theta^0 \wedge \theta^i \ \ \ (i=1,2,3,4).
\]

From the equation \eqref{eq:060},
\begin{equation}
	R_{0ijk} =0 \ \ \ (i,j,k=1,2,3,4)
\end{equation}
holds.
So 
\begin{equation}
	R^-_0 =0
\end{equation}
holds.
Therefore, the integrability condition \eqref{eq:054} is satisfied.

The integrability condition \eqref{eq:053} can be written as
\begin{align}
	&W_{1212}+2W_{1234}+W_{3434} =0, \ \ W_{1313}-2W_{1324}+W_{2424}=0 , \notag \\ 
	&W_{1414} +2W_{1423} +W_{2323} =0, \ \ W_{1213}-W_{1224}+W_{3413} -W_{3424}=0,\notag \\
	&W_{1214} +W_{1223} +W_{3414}+W_{3423}=0, \ \ W_{1314}+W_{1323}-W_{2414} -W_{2423}=0. 
\end{align}

From the definition of the Weyl curvature tensor, we have
\begin{equation}
	W_{ijkl} = R_{ijkl}+\frac{1}{3}(R_{jk}\delta_{il} +R_{il}\delta_{jk}-R_{ik}\delta_{jl}-R_{jl}\delta_{ik}) +\frac{s}{12}(\delta_{ik}\delta_{jl}-\delta_{il}\delta_{jk}).
\end{equation}
So, for example, we have
\begin{align}
	W_{1212}+2W_{1234}+W_{3434} &= R_{1212}-\frac{1}{3}(R_{11}+R_{22}) +\frac{s}{12} \notag \\
	&\ \ +2R_{1234} +R_{3434} -\frac{1}{3}(R_{33}+R_{44}) +\frac{s}{12} \notag \\
	&=-3 +\frac{4}{3} -\frac{1}{3} +4 -3 +\frac{4}{3} -\frac{1}{3}=0, 
\end{align}
\[
	W_{1213}-W_{1224}+W_{3413} -W_{3424} = R_{1213}-R_{1224}+R_{3413}-R_{3424} =0.
\]
hold.
The other equations can be confirmed in the same way.
Therefore, the integrability condition \eqref{eq:053} is satisfied.

Then, we have shown the following:
\begin{theorem}
The real 5-dimensional K-contact manifold, which is the real slice of the complex spacetime associated with the Ren-Wang twistor space, satisfies the integrability conditions in Theorem \ref{thm:50} (by giving a suitable orientation).
\hfill{$\Box$}
\end{theorem}

Let $\mathcal{Z}$ be the CR twistor space associated with the above K-contact manifold $\left( \mathcal{M}^\R, \widetilde{\phi},\widetilde{\xi},\widetilde{\eta},\widetilde{g} \right)$.
In general, a CR twistor space can be identified with the projectivization of the negative spinor bundle over the associated K-contact manifold.
Hence, the CR twistor space $\mathcal{Z}$ coincides with the pull-back $\tau^{-1}(\mathcal{M}^\R) \subset \widetilde{\mathcal{P}}$, where $\tau$ is the map given in the Figure \ref{fig:01}.
As in the case of a 4-dimensional flat spacetime, the map
\begin{equation}
	\eta \colon \mathcal{Z} = \mathcal{M}^\R \times \mathbb{CP}^1 \to \mathcal{P}
\end{equation}
is injective.
The image $\eta \circ \tau^{-1}(\mathcal{M}^\R)$ is given by the equation
\begin{equation}
	\mathrm{Re}(\omega_2) = \frac{\abs{\omega_1}^2 -\abs{\omega_0}^2 -2\mathrm{Re} \left( \omega_0 \omega_1 \bar{\zeta} \right) }{1+\abs{\zeta}^2} ,
\end{equation}
where $(\zeta,\omega_0,\omega_1,\omega_2)$ are coordinates on $W$ (see \eqref{eq:3336}).

Then, we have shown the following:

\begin{proposition}
The Ren-Wang twistor space contains the CR twistor space associated with the real slice as a real submanifold.
\end{proposition}

\section*{Appendix A}

For simplicity, we gave the integrability condition of the tautological distribution $\widetilde{\mathcal{E}}$ as the conditions $R^0_- =0$ and $W^-_-=0$ in Proposition \ref{prop:310}.
In Appendix, to give a more precisely integrability condition, we find a necessary and sufficient condition for the condition
\begin{equation}
	R(v,w)v \in \mathcal{B} \ \ \  \text{for any}\ v,w \in \mathcal{B}.
\end{equation}

Let $(\mathcal{M}^\C,\phi,\xi,\eta,g)$ be a 5-dimensional complex almost contact manifold.
First, we introduce a frame of the tangent bundle, and we identify the tangent space with the inner product space $\C^5$.
Let $(e_0,e_1,e_2,e_3,e_4)$ be a basis of the tangent space, where $e_0$ is $\xi$ and $(e_1,e_2,e_3,e_4)$ is a standard basis of $E \eqdef \mathrm{Ker}\ \eta =\C^4$.
We introduce the basis of $\wedge_-^2 E$,
\begin{equation}
	\phi_1 =\frac{e_1 \wedge e_2 -e_3 \wedge e_4}{\sqrt{2}},\ \phi_2 =\frac{e_1 \wedge e_3 +e_2 \wedge e_4}{\sqrt{2}},\ \phi_3 =\frac{e_1 \wedge e_4 -e_2 \wedge e_3}{\sqrt{2}}.
\end{equation}
In this situation, $\beta$-planes are spanned by two vectors
\begin{align}
	v_1 &= \eta_0 e_1 +\sqrt{-1} \eta_0 e_2 +\eta_1 e_3 +\sqrt{-1} \eta_1 e_4, \notag \\
	v_2 &=  \eta_1 e_1 -\sqrt{-1} \eta_1 e_2 -\eta_0 e_3 +\sqrt{-1} \eta_0 e_4,
\end{align}
where $(\eta_0,\eta_1) \in \C^2 \setminus \{(0,0)\}$.

\begin{proposition*}
The following are equivalent:
\begin{enumerate}
\item For any $\beta$-plane $\mathcal{B}$, if $v,w \in \mathcal{B}$, then $R(v,w)v \in \mathcal{B}$.
\item The almost contact manifold satisfies the curvature conditions
\begin{align}
	W^-_- &=0, \\
	R^0_- &=\begin{pmatrix} \alpha & \delta & -\gamma \\
	\beta & -\gamma & -\delta \\
	\gamma & \beta & \alpha \\
	\delta & -\alpha & \beta \end{pmatrix}, \label{eq:765}
\end{align}
with respect to the dual basis of $(\phi_1,\phi_2,\phi_3)$ and the dual basis of $(e_0 \wedge e_1, e_0 \wedge e_2, e_0 \wedge e_3, e_0 \wedge e_4)$, where $\alpha,\beta,\gamma,\delta$ are functions.
\end{enumerate}
\end{proposition*}
\begin{proof}

The condition $R(v,w)v \in \mathcal{B}$ is equivalent to the conditions
\begin{align}
	g(R(v,w)v,v) &=0, \\
	g(R(v,w)v,w) &=0, \\
	g(R(v,w)v,\xi) &=0.
\end{align}
From similar arguments to (\cite{LM}, LEMMA 3.4), the former two equations are equivalent to the curvature condition
\begin{equation}
	W^-_- =0.
\end{equation}

We need only show that the condition
\begin{equation}
	g(R(v,w)v,\xi) =0
\end{equation}
is equivalent to the condition \eqref{eq:765}.
We denote a $(i,j)$-component of the matrix $R^0_-$ by $r_{ij}$.
The function $r_{ij}$ satisfies the equation
\begin{equation}
	r_{ij} = (R\phi_j, e_0 \wedge e_i).
\end{equation}
Since the equation
\begin{equation}
	v_1 \wedge v_2 = -2 \eta_0 \eta_1 \sqrt{-2} \phi_1 -(\eta_0^2 +\eta_1^2)\sqrt{2} \phi_2 + (\eta_0^2 -\eta_1^2) \sqrt{-2} \phi_3
\end{equation}
holds, we have
\begin{align}
	g(R(v_1,v_2)v_1,e_0) &= (R(v_1 \wedge v_2), e_0 \wedge v_1 ) \notag \\
	&= \eta_0^3 \sqrt{2} (-r_{12}-\sqrt{-1}r_{22} +\sqrt{-1}r_{13}-r_{23}) \notag \\
	&\ \ +\eta_0^2 \eta_1 \sqrt{2} (-2 \sqrt{-1}r_{11} +2r_{21} -r_{32} -\sqrt{-1} r_{42} \notag \\
	&\ \ \ \ \ \ +\sqrt{-1} r_{33} -r_{43}) \notag \\
	&\ \ + \eta_0 \eta_1^2 \sqrt{2} (-2 \sqrt{-1}r_{31} +2r_{41} -r_{12} -\sqrt{-1} r_{22} \notag \\
	&\ \ \ \ \ \ -\sqrt{-1} r_{13} +r_{23}) \notag \\
	&\ \ + \eta_1^3 \sqrt{2} (-r_{32}-\sqrt{-1}r_{42} -\sqrt{-1}r_{33}+r_{43}). \label{eq:764}
\end{align}
If the equation \eqref{eq:764} is $0$ for any $(\eta_0, \eta_1) \in \C^2$, then the component of the curvature $R^-_0$ must be of the form
\begin{equation}
 R^-_0 = \begin{pmatrix} \alpha & \delta & -\gamma \\
	\beta & -\gamma & -\delta \\
	\gamma & \beta & \alpha \\
	\delta & -\alpha & \beta \end{pmatrix} 
\end{equation}
with respect to the dual basis of $(\phi_1,\phi_2,\phi_3)$ and the dual basis of $(e_0 \wedge e_1, e_0 \wedge e_2, e_0 \wedge e_3, e_0 \wedge e_4)$, where $\alpha,\beta,\gamma,\delta$ are functions.
Conversely, if the curvature $R^-_0$ is of the form \eqref{eq:765}, then the equation \eqref{eq:764} is 0 for any $(\eta_0, \eta_1) \in \C^2$.

\end{proof}

Next, We show that the functions $\alpha,\beta,\gamma,\delta$ in \eqref{eq:765} are given by non-diagonal components of the Ricci tensor.

\begin{proposition*}
If the component of the curvature $R^0_-$ is of the form \eqref{eq:765}, then the functions $\alpha,\beta,\gamma,\delta$ are
\begin{equation}
	\alpha = - \frac{1}{3 \sqrt{2}} R_{20},\ \ \beta = \frac{1}{3 \sqrt{2}} R_{10},\ \ \gamma = \frac{1}{3 \sqrt{2}} R_{40},\ \ \delta = -\frac{1}{3 \sqrt{2}} R_{30},
\end{equation}
i.e. $R^0_-$ is of the form 
\begin{equation}
	R^-_0 = \frac{1}{3 \sqrt{2}} \begin{pmatrix} -R_{20} & -R_{30} & -R_{40} \\
	R_{10} & -R_{40} & R_{30} \\
	R_{40} & R_{10} & -R_{20} \\
	-R_{30} & R_{20} & R_{10} \end{pmatrix},
\end{equation}
where $R_{ij}$ is the Ricci tensor with respect to the frame $(e_0,e_1,e_2,e_3,e_4)$.
\end{proposition*}
\begin{proof}
From the equation \eqref{eq:765}, 
\begin{align}
	\frac{1}{\sqrt{2}}(R_{1201} -R_{3401}) - \frac{1}{\sqrt{2}}(R_{1304} + R_{2404}) +\frac{1}{\sqrt{2}}(R_{1403} -R_{2303}) = 3 \alpha \label{eq:766}
\end{align}
holds.
On the other hand, from the Bianchi identity and $R_{2202}= R_{2000}=0$, we have
\begin{align}
	&\ \ (R_{1201} -R_{3401}) - (R_{1304} + R_{2404}) + (R_{1403} -R_{2303}) \notag \\ 
	&= -(R_{2101} +R_{2303} + R_{2404} ) - (R_{0134} + R_{0413}+R_{0341}) \notag \\
	&= -R_{20}.
\end{align}
Hence, we get
\begin{equation}
	\alpha = - \frac{1}{3 \sqrt{2}} R_{20}.
\end{equation}
The other functions are given in the same way.
Then, we get
\begin{equation}
	\beta = \frac{1}{3 \sqrt{2}} R_{10},\ \ \gamma = \frac{1}{3 \sqrt{2}} R_{40},\ \ \delta = -\frac{1}{3 \sqrt{2}} R_{30}. 
\end{equation}
\end{proof}

\section*{Appendix B}

In this section, we show that in the case of the Ren-Wang twister space, the one form we obtained coincides with the one form constituting the torsional Galilean structure obtained by Gundry\cite{G}.

First, we explain the result of Gundry\cite{G}.
Let $\mathcal{P}_f$ be a holomorphic $\C^3$-bundle over $\mathbb{CP}^1$ given by the patching
\begin{align}
	\hat{T} &= T + \epsilon f(\Omega^A,\lambda), \\
	\hat{\Omega}^A &=\lambda^{-1} \Omega^A,\ \ \ (A=0,1)
\end{align}
where $\lambda$ is an inhomogeneous coordinate on $\mathbb{CP}^1$ and $f$ a holomorphic function.
The manifold $\mathcal{P}_f$ is a 4-dimensional twistor space.
Let $\mathcal{M}^\C$ be the set of twistor lines in $\mathcal{P}_f$.
Gundry constructed a torsional Galilean structure on $\mathcal{M}^\C$.
The 1-form $\theta_1$ constituting the Galilean structure is given by
\begin{equation}
	\theta_1 = ds+ \epsilon \phi_{0,A} dx^{A1^\prime}, \label{eq:6007}
\end{equation}
where $(x^{AA^\prime},s)$ is suitable coordinates on $M$.
The function $\phi_{0,A}$ is given by
\begin{equation}
	\phi_{0,A} = \frac{1}{2 \pi \sqrt{-1}} \oint f(\Omega^A \rvert, \lambda) \lambda^{-1}\ d\lambda,
\end{equation}
where $\rvert$ is the restriction to a twistor line.

Next, we calculate \eqref{eq:6007} in the case of the Ren-Wang twistor space.
In this case, we set
\begin{align}
	T &= \omega_2, \notag \\
	\Omega^A &=\omega_A \ \ \ (A=0,1), \notag \\
	\lambda &= \zeta \notag \\
	\epsilon &=1, \notag \\
	f &= 2 \lambda^{-1} \Omega^0 \Omega^1, \notag \\
	x^{AA^\prime} &= y_{A \overline{A\prime}}, \notag
\end{align}
where an integer $\overline {A} \in \{ 0,1 \}$ satisfies 
\[
	\overline{A} \equiv A+1 \ \ \ (\mathrm{mod}\ 2).
\]
Since the coordinate $s$ is $T \rvert(\lambda=0)$, we have
\begin{equation}
	s= t -y_{00^\prime}y_{11^\prime} - y_{10^\prime}y_{01^\prime}.
\end{equation}
From the direct calculation, we have
\begin{align}
	\phi_{0,A} &= \frac{1}{2\pi \sqrt{-1}} \oint (2y_{\overline{A}1^\prime}\lambda^{-1} + 2y_{\overline{A}1^\prime} \lambda^{-2})\ d\lambda \notag \\
	&= 2y_{\overline{A}1^\prime}.
\end{align}
Thus, we have
\begin{align}
	\theta_1 &= ds+ \epsilon \phi_{0,A} dx^{A1^\prime} \notag \\
	&=dt -y_{11^\prime} dy_{00^\prime} -y_{00^\prime} dy_{11^\prime} -y_{10^\prime} dy_{01^\prime} -y_{01^\prime} dy_{10^\prime} \notag \\
	&\ \ \ + 2y_{11^\prime} dy_{00^\prime} +2 y_{01^\prime} dy_{10^\prime} \notag \\
	&=dt  +y_{11^\prime} dy_{00^\prime}   -y_{10^\prime} dy_{01^\prime} +y_{01^\prime} dy_{10^\prime} -y_{00^\prime} dy_{11^\prime}. 
\end{align}
Therefore, the 1-form $\theta_1$ coincides with the 1-form $\theta$ given in \eqref{eq:6008}.

Graduate Major in Mathematics, School of Science, Tokyo Institute of Technology, Meguro, Tokyo \\
e-mail:teruya.m.aa@m.titech.ac.jp

\end{document}